\documentclass[12pt,reqno]{amsart}
\usepackage{amsmath,amssymb,latexsym,textcomp,mathrsfs}
\usepackage[all]{xy}
\usepackage{graphicx}
\usepackage{bm,amsmath, amsthm, amssymb, amsfonts}
\usepackage{times}
\usepackage[english]{babel}

\setbox0=\hbox{$+$}
\newdimen\plusheight
\plusheight=\ht0
\def\+{\;\lower\plusheight\hbox{$+$}\;}

\setbox0=\hbox{$-$}
\newdimen\minusheight
\minusheight=\ht0
\def\-{\;\lower\minusheight\hbox{$-$}\;}

\setbox0=\hbox{$\cdots$}
\newdimen\cdotsheight
\cdotsheight=\plusheight
\def\cds{\lower\cdotsheight\hbox{$\cdots$}}

\numberwithin{equation}{section}
\theoremstyle{plain}
\newtheorem{theorem}{Theorem}[section]
\newtheorem{lemma}{Lemma}[section]

\newtheorem{corollary}{Corollary}[section]
\newtheorem{definition}{Definition}[section]

\newtheorem{remark}{Remark}[section]

\newtheorem{note}{Note}[section]
\newtheorem{example}{Example}[section]

\newenvironment{nouppercase}{%
\renewcommand{\uppercasenonmath}[1]{}}{}
\newcommand{\Keywords}[1]{\par\noindent{\small{Keywords and phrases}: #1}}
\newcommand{\AMS}[1]{\par\noindent{\small{AMS Subject Classification (2010)}: #1}}

\begin{document}

\title{NEW SEPARATION AXIOMS IN GENERALIZED BITOPOLOGICAL SPACES}

 \author{Amar Kumar Banerjee$^1$}
 \author{Jagannath Pal$^2$}
 \newcommand{\acr}{\newline\indent}
\address{{1} Department of Mathematics, The University of Burdwan, 
 West Bengal, India.
 Email:
akbanerjee@math.buruniv.ac.in, akbanerjee1971@gmail.com\acr
 {2} Department of Mathematics, The University of Burdwan, 
 West Bengal, India. 
 Email: jpalbu1950@gmail.com\\}

 \maketitle
 \begin{abstract}          
 Here we have studied on the ideas of $g_{\mu_i}$ and $\lambda_{\mu_i}$-closed sets  with respect to  ${\mu_j}(i,j=1,2,i\not=j)$ and pairwise $ \lambda $-closed sets   in a generalized bitopological space $ (X,\mu_1, \mu_2) $. We have also investigated the properties on some new separation axioms namely pairwise $ T_\frac{1}{4}$, pairwise $T_\frac{3}{8}$, pairwise $  T_\frac{5}{8} $ and have established their mutual relations with pairwise $ T_0 $, pairwise $ T_\frac{1}{2} $ and pairwise $ T_1 $. We have also shown that under certain conditions these axioms are equivalent.
\end{abstract}

\begin{nouppercase}
\maketitle
\end{nouppercase}

\let\thefootnote\relax\footnotetext{
\AMS{Primary 54A05, 54D10}

}

\section{\bf Introduction}
\label{sec:int}
A. D. Alexandroff \cite{AD} introduced the idea of $\sigma$-space (or Alexandroff space) in 1940 generalizing the idea of a topological space where only countable unions of open sets were taken to be open. The notion of a topological space was generalized to a bitopological space by J.C. Kelly \cite{JK} in 1963. Later many works on bitopological spaces were carried out \cite{RL, RI}. In 2001 the idea of space was used by B. K. Lahiri and P. Das \cite{LK} to generalize this idea of bitopological space to a bispace. Later many works were done in bispaces \cite{AC, KP, BS, BP, AP, AR}.

Recently a generalized topology $\mu $ \cite{MS} has been defined on a nonempty set $ X $ as the collection of subsets of $ X $ which is closed under arbitrary unions and satisfies  the condition that the null set $ \emptyset $ belongs to $ \mu $. The generalized topological space (in short a GT-space) may be denoted by $ (X, \mu) $. M. S. Sarsak \cite{MS} studied $ g_\mu $-closed sets on a generalized topological space and introduced some new separation axioms namely $ \mu$-$T_\frac{1}{4},  \mu$-$T_\frac{3}{8} $ and $ \mu$-$T_\frac{1}{2} $ axioms using the idea of $ \lambda_\mu $-closed sets and investigated their properties and relations among the axioms. We have studied these ideas in more general structure of a generalized bitopological space using the idea of $g_{\mu_i}$ and $\lambda_{\mu_i}$-closed sets with respect to  ${\mu_j}$. We have also studied the idea of pairwise $ \lambda $-closed sets and obtain some new separation axioms viz. pairwise $ T_\frac{1}{4}$, pairwise $T_\frac{3}{8}$, pairwise $  T_\frac{5}{8} $ and their relations with pairwise $ T_0 $, pairwise $ T_\frac{1}{2} $ and pairwise $ T_1 $ and investigate how far several results as valid in \cite{MS} are affected. Some examples are substantiated where necessary in support of these results.

 \section{\bf Generalized Bitopological Space}
 
 \label{sec:pre}

In a GT-space $ (X, \mu) $, members of $ \mu $ are called $ \mu $-open sets and complement of $ \mu $-open sets are called $ \mu $-closed sets. The definition of $ \mu $-closure,  $ \mu $-interior and $ \mu $-limit point of a set $ A $  are similar as in the case of a topological space and is denoted respectively by $ \overline{A_{\mu}} $,  $ Int_{\mu}(A) $ and $ \mu $-limit. Clearly $ \overline{A_{\mu}} $ and $ Int_{\mu}(A) $ are  respectively $ \mu $-closed set and $ \mu $-open set. So $ A $ is $ \mu $-closed iff $ A= \overline{A_{\mu}}$ and $ A $ is $ \mu $-open iff $ A= Int_{\mu}(A)$.\\

\begin{definition} \cite{JK}: A nonempty set $ X $ on which are defined two arbitrary topologies $ P $ and $ Q $ is said to be a bitopological space and is denoted by $ (X, P, Q) $.
\end{definition}

\begin{definition}\label{1} (cf.\cite{JK}): Suppose $ X $ is a nonempty set. If $ \mu_1 $ and $ \mu_2 $ be two collection of subsets of $ X $ such that $ (X, \mu_1) $ and $ (X, \mu_2) $ are two GT-spaces, then $ X $ is said to be a  generalized bitopological space, briefly GBT-space and is denoted by $ (X, \mu_1, \mu_2)$.
\end{definition}

Throughout the paper $ X $ stands for a GBT-space with generalized topologies $ \mu_1 $ and $ \mu_2 $ and sets are always subsets of $ X $ and unless otherwise stated. The letters $ R,  P$ and $ Q $ stand respectively for the set of real numbers, the set of irrational numbers and the set of rational numbers.
\begin{definition}\label{2} (cf.\cite{JK}): A  GBT-space $(X, \mu_1, \mu_2)$ is called pairwise $T_0$ if for any two distinct points  $x, y\in X $, either there is a $ \mu_1 $-open set $ U $ such that $ x\in U $ and $ y\not\in U $ or, there is a $ \mu_2 $-open set $ V $ such that $ y\in V $ and $ x\not\in V $.
\end{definition}
\begin{definition}\label{3} (cf.\cite{RL}: A  GBT-space $(X, \mu_1, \mu_2)$ is said to be pairwise $T_1$ if for any $x, y\in X, x\not=y $,  there exist open sets  $ U \in\mu_1, V  \in\mu_2$ such that $x \in U,  y \not \in U,  y\in V , x\not\in V$.
\end{definition}
\begin{definition}\label{4} (cf.\cite{RL}): A  GBT-space $(X, \mu_1, \mu_2)$ is said to be pairwise $R_0$ if for each $ \mu_i $-open set  $ G, x\in G $ implies $ \overline{\{x\}_{\mu_j}} \subset G ;i,j=1,2; i\not=j $. 
\end{definition}

\begin{definition}\label{6}(cf.\cite{MS}): In a GT-space $ (X,  \mu), A\subset X $. We denote $ A_{\mu}^\wedge = \cap \{U\in \mu: A\subset U\}$ and 
$A_{\mu}^\vee =\cup \{F: X-F \in \mu :F\subset A\}$. Note that if there is no $ \mu $-open set containing $ A $, then we define $ A_{\mu}^\wedge = X $ and if there is no $ \mu $-closed set contained in $ A $, then $ A_{\mu}^\vee = \emptyset $. Again, $ A $ is called a $ \wedge_{\mu}$-set if $ A= A_{\mu}^\wedge$ and $ A $ is called a $ \vee_{\mu}$-set if $A= A_{\mu}^\vee $.\end{definition}

\begin{lemma}\label{7} (cf.\cite{MS}):  Let $A , B$ be subsets of $(X, \mu)$ then  

(1) $\emptyset_{\mu}^\wedge = \emptyset $, \qquad      $\emptyset_{\mu}^\vee = \emptyset$, \qquad$X_{\mu}^\wedge = X$,\qquad$X_{\mu}^\vee = X $

(2) $A \subset A_{\mu}^\wedge$,   \qquad $A_{\mu}^\vee \subset A $

(3)  $(A_{\mu}^\wedge)^\wedge = A_{\mu}^\wedge$,\qquad  $(A_{\mu}^\vee)^\vee = A_{\mu}^\vee $.

(4)  If $A \subset B $ then $ A_{\mu}^\wedge \subset B_{\mu}^\wedge$.

(5) If $A \subset B $ then $ A_{\mu}^\vee \subset B_{\mu}^\vee $.
\end{lemma}
The proof are simple and so are omitted.
\quad

\quad

\quad

\section{\bf  $  \textit \quad \textit  g_{\mu_i}$-closed sets w.r.to $ \mu_j $  and  pairwise $\textit T_\frac{1}{2} $ GBT-space}

\begin{definition}\label{8} (cf. \cite {TA} :) In a GBT-space $(X, \mu_1, \mu_2) $, a set $ A $ is said to be $g_{\mu_i}$-closed set with respect to $ \mu_j $ if and only if $ \overline{A_{\mu_i}}\subset U $ for any $ \mu_j $-open set $ U\supset A $. Again $ A $ is called $g_{\mu_i}$-open with respect to $ \mu_j $ if $ X -A $ is $g_{\mu_i}$-closed with respect to $ \mu_j $ or equivalently, $ F\subset Int_{\mu_i}(A) $ whenever $ F\subset A $ and $ F $ is $ \mu_j $-closed, $ i, j=1,2; i\not=j$. \end{definition}

\begin{remark}\label{9} Clearly a set $ A $ is $g_{\mu_i}$-closed with respect to $ \mu_j $ if and only if $ \overline{A_{\mu_i}}\subset A_{\mu_j}^\wedge $. By definition, every $ \mu_i $-closed set is $g_{\mu_i}$-closed with respect to  $ \mu_j $. But the converse may not be true which is  shown in Example \ref{11}. However, converse is true in a GBT-space $(X, \mu_1, \mu_2)$ if moreover the set $ A $ is $ g_{\mu_i} $-closed with respect to $ \mu_j $ which is also $ \mu_j $-open since $ A_{\mu_j}^\wedge=A ,i,j=1,2; i\not=j $. \end{remark}

\begin{note}\label{10} If a set $A$ is  $ \wedge _{\mu_j}$-set i.e. $ A_{\mu_j}^\wedge=A $ then $ A $ is $g_{\mu_i}$-closed  with respect to $ \mu_j $   if and only if $ A $ is $ \mu_i $-closed since $ \overline{A_{\mu_i}}\subset A_{\mu_j}^\wedge =A $; converse part is obvious. In particular $(A_{\mu_j}^\wedge)_{\mu_j}^\wedge = A_{\mu_j}^\wedge$ and so $A_{\mu_j}^\wedge$ is a $\wedge_{\mu_j}$-set.  Therefore $A_{\mu_j}^\wedge$ is $g_{\mu_i}$-closed with respect to $ \mu_j $ if and only if $A_{\mu_j}^\wedge$ is $ \mu_i $-closed, $ i, j=1,2; i\not=j $. \end{note}

\begin{example}\label{11} Suppose $ X=\{a, b, c\}, \mu_1 = \{\emptyset, \{c\}, \{a, c\}\}, \mu_2=\{\emptyset, \{b\}, \{a,b\}\} $. Then  $(X, \mu_1, \mu_2) $ is a generalized bitopological space but not a bitopological space. Here the set $ \{a\} $ is $ g_{\mu_1} $-closed with respect to $ \mu_2 $ but it is not $ \mu_1 $-closed. Also the set $ \{b, c\} $ is $ g_{\mu_2} $-closed with respect to $ \mu_1 $ but it is not $ \mu_2 $-closed.
\end{example}

\begin{theorem}\label{12}
 In a GBT-space $(X, \mu_1, \mu_2) $, if a set $ A $ of $ X $ is $ g_{\mu_i} $-closed with respect to $ \mu_j $, then $ \overline{A_{\mu_i}}- A $ does not contain any non-void  $ \mu_j $-closed set, $ i, j=1,2; i\not=j $.
\end{theorem}
\begin{proof}
Suppose $ A$ is a $ g_{\mu_i} $-closed set with respect to $ \mu_j $, then $ \overline{A_{\mu_i}}\subset U $ when $ U $ is $ \mu_j $-open containing $ A $. Again suppose $ F\subset \overline{A_{\mu_i}}- A, F $ is  $ \mu_j $-closed. Therefore $ X-F $ is $ \mu_j $-open and $ A\subset X-F $. Then by Definition, $ \overline{A_{\mu_i}}\subset X-F$ which implies that $ F\subset X-\overline{A_{\mu_i}} $. Again $ F\subset \overline{A_{\mu_i}} $. Therefore $ F\subset (X-\overline{A_{\mu_i}})\cap\overline{A_{\mu_i}}=\emptyset $. 
\end{proof}

But the converse is not always true as seen from the Example \ref{13}

\begin{example} \label{13} Suppose  $ X = \{a, b, c\}, \mu_1=\{\emptyset, \{a\}, \{a, b\}\}$ and $\mu_2=\{\emptyset, \{b\}, \{b, c\}\}$.  Then $ (X, \mu_1, \mu_2) $ is a generalized bitopological space but not a bitopological space.  Take a set $ \{b\} $, then $ \overline{\{b\}_{\mu_1}}-\{b\}=\{c\} $ which does not contain any non-void $ \mu_2 $-closed set. But $ \overline{\{b\}_{\mu_1}}=\{b, c\}\not\subset \{b\} $, a $ \mu_2 $-open set and hence $ \{b\} $ is not $ g_{\mu_1} $-closed with respect to $ \mu_2 $. 
\end{example}

\begin{remark} In a GBT-space $(X, \mu_1, \mu_2) $, union and intersection of two $ g_{\mu_i} $-closed sets with respect to $ \mu_j $ may not be $ g_{\mu_i} $-closed  with respect to $ \mu_j, i, j=1,2; i\not=j $ as shown by the Examples \ref{14}. 
\end{remark}

\begin{example}\label{14}
Suppose  $ X = \{a, b, c, d\},  \mu_1=\{\emptyset,  \{a, d\}, \{c, d\}, \{a, c, d\}\}$ and $\mu_2=\{\emptyset, \{d\}, \{a, c, d\}\}$.  Then $ (X, \mu_1, \mu_2) $ is a GBT-space but not a bitopological space.  Take  $A=\{a, d\} $ and $ B=\{c, d\} $, then clearly $ A $ and $ B $ are $ g_{\mu_1} $-closed sets with respect to $ \mu_2 $. But $ A\cap B =\{d\}$ and $ A\cup B=\{a, c, d\} $ are not $ g_{\mu_1} $-closed with respect to $ \mu_2 $.
\end{example}

\begin{theorem} In a GBT-space $(X, \mu_1, \mu_2) $ union of two $ g_{\mu_i} $-closed sets with respect to $ \mu_j $ is $ g_{\mu_i} $-closed  if union of two $ \mu_i $-closed sets is $ g_{\mu_i} $-closed with respect to $ \mu_j,i,j=1,2;i\not=j $. \end{theorem}
\begin{proof}
Let $ A, B $ be two $ g_{\mu_i} $-closed sets with respect to $ \mu_j $ and $ A\cup B \subset U, U $ is $ \mu_j $-open. Therefore $ A\subset U, B\subset U $ implies $ \overline{A_{\mu_i}}\subset U $ and $ \overline{B_{\mu_i}}\subset U $, so $ \overline{A_{\mu_i}}\cup \overline{B_{\mu_i}} \subset U $. $ \overline{A_{\mu_i}}, \overline{B_{\mu_i}} $ are $ \mu_i $-closed sets. According to the given condition $ \overline{A_{\mu_i}}\cup \overline{B_{\mu_i}}$ is $ g_{\mu_i} $-closed with respect to $ \mu_j $, then $ \overline{( \overline{A_{\mu_i}}\cup \overline{B_{\mu_i}} )_{\mu_i}} \subset U $. Now $ A\cup B \subset \overline{A_{\mu_i}}\cup \overline{B_{\mu_i}} $ implies $ \overline{(A\cup B)_{\mu_i}}\subset \overline{( \overline{A_{\mu_i}}\cup \overline{B_{\mu_i}} )_{\mu_i}} \subset U$. 
\end{proof}

\begin{definition}  In a GT-space $(X, \mu_i) $, two sets $ A, B $ are said to be $ \mu_i $-weakly separated if there are two $ \mu_i$-open sets $U, V $ such that $A\subset U, B\subset V $ and $ A\cap V=B\cap U=\emptyset,i=1,2$.
\end{definition}

\begin{theorem} In a GBT-space $(X, \mu_1, \mu_2) $, union of two $ \mu_j $-weakly separated $ g_{\mu_i} $-open sets with respect to $ \mu_j $ is $ g_{\mu_i} $-open with respect to $ \mu_j; i, j=1,2; i\not=j $.\end{theorem}
\begin{proof}
Suppose $ A_1,A_2 $ are two $ \mu_j $-weakly separated $ g_{\mu_i} $-open sets with respect to $ \mu_j $. Then there are $ \mu_j $-open sets $ U_1, U_2 $ such that $ A_1\subset U_1, A_2\subset U_2 $ and $A_1\cap U_2=A_2\cap U_1=\emptyset$. Let $ F_1=U_1^c, F_2=U_2^c $. Then $ F_1, F_2 $ are $ \mu_j $-closed sets and $ A_1\subset F_2, A_2\subset F_1 $. Again if $ P_n\subset A_n, n=1,2, P_n $ is $ \mu_j $-closed then $ P_n\subset Int_{\mu_i}(A_n) $ by Definition \ref{8} since  $ A_1,A_2 $ are two $g_{\mu_i} $-open sets with respect to $ \mu_j $.  Here $ Int_{\mu_i}(A_n) $ is $\mu_i $-open set, so $ Int_{\mu_i}(A_1)\cup Int_{\mu_i}(A_2) $ is a $\mu_i $-open set and $ Int_{\mu_i}(A_1)\cup Int_{\mu_i}(A_2) \subset   A_1\cup A_2 $. Let $ F\subset A_1\cup A_2 $ and $ F $ is $ \mu_j $-closed.  Now $ F= F\cap (A_1\cup A_2 )=(F\cap A_1)\cup (F\cap A_2)\subset (F\cap F_2)\cup (F\cap F_1) $ where $(F\cap F_2)$ and $ (F\cap F_1)$ are  $ \mu_j $-closed sets. Further $ (F\cap F_1)\subset (A_1\cup A_2)\cap F_1= (F_1\cap A_1)\cup (F_1\cap A_2)=\emptyset\cup (F_1\cap A_2)\subset A_2$ and so $ (F\cap F_1)\subset Int_{\mu_i}(A_2)$. Similarly, $ (F\cap F_2)\subset Int_{\mu_i}(A_1)$. So $ F\subset Int_{\mu_i}(A_1)\cup Int_{\mu_i}(A_2)\subset Int_{\mu_i}(A_1\cup A_2) $. Hence the result follows. 
\end{proof}

\begin{theorem}\label{15} 
A GBT-space $(X,  \mu_1, \mu_2)$  is pairwise $T_0$ if and only if for any pair of distinct points $x, y \in X $,  there is a set $A$ which contains one of them only  such that $ A $ is either $ \mu_i $-open or $ \mu_j $-closed, $ i, j=1,2 $.\end{theorem}
\begin{proof}    
Let the condition hold and $ A $ be a $ \mu_i $-open set containing $ x $ only. Then clearly the GBT-space is pairwise $ T_0 $. If $ x\in A $ but $ A $ is $ \mu_j $-closed, then $ X-A $ is $ \mu_j $-open and $ y\in X-A, x\not\in X-A $. Hence in this case also the GBT-space is pairwise $ T_0$.

Conversely, suppose that $(X,  \mu_1, \mu_2)$  is pairwise $T_0$ and  $x, y \in X, x\not=y $ and there exists a $ \mu_i $-open set $ A $ containing one of $ x, y, $ say $ x $ but not $ y $.Then it is done. There may be the case that $ x\in A $ and $ y \not\in A $, but $ A $ is not a $ \mu_i $-open set, then $ y\in X-A, x\not\in X-A$ and $ X-A $ is $ \mu_j $-open. So $ A $ is $ \mu_j $-closed. 
\end{proof}

\begin{remark}\label{16} If $ (X, \mu_1) $ or $ (X, \mu_2) $ is $T_0$ generalized topological space, then $(X,  \mu_1, \mu_2)$  is pairwise $T_0$. It can be easily shown that pairwise $ T_1 $ is pairwise $T_0$. But the converses may not be true as seen from the Example \ref{17} given below.\end{remark}

\begin{example}\label{17} Suppose  $ X = \{a, b, c\},  \mu_1=\{\emptyset,  \{a\}\}$ and $\mu_2=\{\emptyset, \{b\}\}$.  Then $ (X, \mu_1, \mu_2) $ is a generalized bitopological space but not a bitopological space. It is pairwise $T_0$ but not pairwise $ T_1 $. Also $ (X, \mu_1) $ or $ (X, \mu_2) $ are not $T_0$ generalized topological space.
\end{example}
\begin{theorem}\label{18}  If $ (X,  \mu_1, \mu_2) $ is  pairwise $ T_0 $, then for every pair of distinct points $ p,q\in X $, either $ p\not\in \overline{\{q\}_{\mu_1}} $ or $ q\not\in \overline{\{p\}_{\mu_2}} $. Proof is simple, so is omitted. \end{theorem}

\begin{definition}\label{19} A GBT-space $(X, \mu_1,\mu_2) $ is said to be pairwise symmetric if for any $ x,y\in X, x\in \overline{\{y\}_{\mu_i}}\Longrightarrow y\in \overline{\{x\}_{\mu_j}}; i, j=1,2; i\not=j$. \end{definition}

\begin{theorem}\label{20}Pairwise $ T_0 $ and pairwise symmetric GBT-space is pairwise $ T_1 $. \end{theorem}
\begin{proof}
Let $(X, \mu_1,\mu_2) $ be pairwise symmetric and pairwise $ T_0 $ and $ a, b\in X, a\not =b $. Since $(X, \mu_1,\mu_2) $ is  pairwise $ T_0 $ then by Theorem \ref{18}, either $ a\not\in\overline{\{b\}_{\mu_1}} $ or $ b\not\in\overline{\{a\}_{\mu_2}} $. So let $ a\not\in\overline{\{b\}_{\mu_1}} $. Now if $ b\in\overline{\{a\}_{\mu_2}} $ then by pairwise symmetryness, $ a\in\overline{\{b\}_{\mu_1}} $ which contradicts $ a\not\in\overline{\{b\}_{\mu_1}} $. Therefore $ b\not\in\overline{\{a\}_{\mu_2}} $. Since $ a\not\in\overline{\{b\}_{\mu_1}} $, there is a $ \mu_1 $-closed set $ F $ such that $ b\in F $ and $  a\not\in F $, so $ a\in X-F $, a $ \mu_1 $-open set and $ b\not\in X-F $. Again since $ b\not\in\overline{\{a\}_{\mu_2}} $, there is a $ \mu_2 $-closed set $ P $ such that $ a\in P $ and $  b\not\in P $. So $ b\in X-P $, a $ \mu_2 $-open set and $ a\not\in X-P $. Hence  $(X, \mu_1,\mu_2) $ is pairwise $ T_1 $.
\end{proof}

\begin{theorem}\label{21} If a GBT-space $(X,  \mu_1, \mu_2)$  is pairwise $T_1$, then each singleton is either $ \mu_1 $-closed or  $ \mu_2 $-closed.\end{theorem}
\begin{proof}
Let $(X,  \mu_1, \mu_2)$  be pairwise $T_1$  and $ x\in X $. Suppose $ y\not\in \{x\} $,  then $ x\not=y $. For  pairwise $T_1$ there may arise two cases.

Case I: There exist $ \mu_1 $-open set $ U $ and  $ \mu_2 $-open set $ V $ such that $x \in U,  y \not \in U,  y\in V , x\not\in V$. Then $ y $ can not be a $ \mu_2 $-limit point of $ \{x\} $, hence $ \{x\} $ is $ \mu_2 $-closed.

Case II: Let there exist  $ \mu_1 $-open set $ U $ and  $ \mu_2 $-open set $ V $ such that $y \in U,  x \not \in U,  x\in V , y\not\in V$. So $ y $ can not be a $ \mu_1 $-limit point of $ \{x\} $, hence $ \{x\} $ is $ \mu_1 $-closed.
\end{proof}

The converse of the Theorem \ref{21} may not be true as shown by the Example \ref{22}.
 
\begin{example}\label{22}
 Suppose  $ X = \{a, b, c\},  \mu_1=\{\emptyset, \{b, c\}, \{c, a\}, \{a, b, c\}\}$ and $\mu_2=\{\emptyset, \{a, b\}\}$.  Then $ (X, \mu_1, \mu_2) $ is a generalized bitopological space but not a bitopological space. Here each singleton is either $ \mu_1 $-closed or  $ \mu_2 $-closed. Now take $ a, b\in X, a\not=b;a, b $ both belong to $ \mu_2 $-open set $ \{a, b\} $. So the GBT-space is not pairwise $ T_1 $.
\end{example}
Example \ref{25} shows Pairwise $ T_1 $  may not imply each singleton is both  $ \mu_1 $ and  $ \mu_2 $-closed.
\begin{note}\label{23} However, converse of the Theorem \ref{21} holds in a GBT-space $(X,  \mu_1, \mu_2)$, if each singleton is both  $ \mu_1 $ and  $ \mu_2 $-closed which can be checked easily. \end{note}
\begin{remark}\label{24}   Example \ref{25} shows that pairwise $ T_1 $ GBT-space $(X,  \mu_1, \mu_2)$ may not imply individual $ T_1 $ GT-space $ (X, \mu_i), i=1,2 $. However any one of $ (X, \mu_1) $ and $ (X, \mu_2) $ may be $ T_1 $ without being pairwise $ T_1 $ as shown in Example\ref{26} .\end{remark}

\begin{example}\label{25} Suppose  $ X = \{a, b\},  \mu_1=\{\emptyset, \{a\}\}$ and $\mu_2=\{\emptyset, \{b\}\}$.  Then $ (X, \mu_1, \mu_2) $ is a generalized bitopological space but not a bitopological space. Clearly the GBT-space is pairwise $ T_1 $ but $ (X, \mu_i), i=1,2, $ is not individual $ T_1 $ generalized topological space.
\end{example}

\begin{example}\label{26} Suppose  $ X = \{a, b, c\},  \mu_1=\{\emptyset, \{a, b\}, \{b, c\}, \{c, a\}, \{a, b, c \}\}$ and $\mu_2=\{\emptyset, \{a, b\}\}$.  Then $ (X, \mu_1, \mu_2) $ is a generalized bitopological space but not a bitopological space. Here $ (X, \mu_1) $ is a $ T_1 $ generalized topological space. Now if we take $ a, b\in X, a\not=b $, then $ a, b$ both belong to $ \{a, b\} $, a $ \mu_2 $-open set, which contradicts the definition of pairwise $ T_1 $, so $ (X, \mu_1, \mu_2) $ is not pairwise $ T_1 $. \end{example}

\begin{definition}\label{27} (c.f.\cite{TA}) :  A GBT-space $ (X, \mu_1, \mu_2) $ is said to be pairwise $ T_\frac{1}{2}$ if every $ g_{\mu_i}$-closed  set with respect to $ \mu_j $ is $ \mu_i $-closed, $ i, j=1,2; i\not=j $.\end{definition}

\begin{theorem}\label{28}  A GBT-space $ (X, \mu_1, \mu_2)$ is  pairwise $ T_\frac{1}{2}$ if and only if a singleton is not $ \mu_j $-closed implies that it is $ \mu_i $-open  $ i, j=1,2; i\not=j $. \end{theorem} 
\begin{proof}
Suppose $ (X, \mu_1, \mu_2) $ is  pairwise $ T_\frac{1}{2}$ and  $ x\in X $ and $ \{x\} $ is not $ \mu_j $-closed. Then $ X-\{x\}=A $ (say), is not $ \mu_j $-open. But $ X $ is the only possible $ \mu_j $-open set containing $ A $ and so $ A_{\mu_j}^\wedge= X $, therefore $ \overline{A_{\mu_i}}\subset X $. Hence $ A $ is $ g_{\mu_i}$-closed  with respect to $ \mu_j $ by Remark \ref{9} and so $ A $ is $ \mu_i $-closed since $ X $ is pairwise $ T_\frac{1}{2}$. This implies $ \{x\} $ is $ \mu_i $-open.

Conversely, suppose that a subset $ A $ of $ X $ is $ g_{\mu_i} $-closed with respect to $ \mu_j $. We claim that $ A $ is $ \mu_i $-closed. Let $ p\in \overline{A_{\mu_i}} $ and $ \{p\} $ is not $ \mu_j $-closed. Then $ \{p\} $ is $ \mu_i $-open. So $ p $ can not be a $ \mu_i $-limit point of $ A $, so $ p\in A $. Hence $ A $ is $ \mu_i $ closed. Now let $ \{p\} $ be $ \mu_j $-closed, then it follows from Theorem \ref{12} that $ \{p\}\not\subset \overline{A_{\mu_i}} - A $ and so $ p\in A $. Hence $ A $ is $ \mu_i $ closed.
\end{proof}

\begin{corollary}\label{29} A GBT-space $ (X, \mu_1, \mu_2) $ is pairwise $ T_\frac{1}{2}$ if and only if a singleton is not $ \mu_i $-open, then it is $ \mu_j $-closed  $ i, j=1,2; i\not=j $. Proof is simple so is omitted. \end{corollary}

 \begin{theorem}\label{30} In a GBT-space $ (X, \mu_1, \mu_2) $, each singleton   is either $ \mu_i $-open or $ \mu_j $-closed if and only if each $ g_{\mu_i} $-closed set with respect to $ \mu_j $ is $ \mu_i $-closed, $ i, j=1, 2; i\not=j $.\end{theorem}
 \begin{proof}
 Suppose each singleton is either $ \mu_i $-open or $ \mu_j $-closed. Let a subset $ B $ be $ g_{\mu_i} $-closed with respect to $ \mu_j $ and $ p\in \overline{B_{\mu_i}} $. Then either $p\in B  $ or $ p $ is the $ \mu_i $-limit point of $ B $. Suppose $ \{p\} $ is $ \mu_i $-open, so $ p $ cannot be a $ \mu_i $ limit point of $ B $, then $p\in B  $. Hence $ B $ is $ \mu_i $-closed. Next let  $ \{p\} $ be $ \mu_j $-closed. Then it follows from Theorem \ref{12} that $ \{p\}\not\subset \overline{B_{\mu_i}}-B $ and so $p\in B  $. Hence $ B $ is $ \mu_i $-closed.
 
 Conversely, suppose the conditions hold and let $ x\in X $ and $ \{x\} $ is not $ \mu_j $-closed. Hence $ X-\{x\}=A $ (say), is not $ \mu_j $-open and so $ A_{\mu_j}^\wedge= X $, therefore $ \overline{A_{\mu_i}}\subset  A_{\mu_j}^\wedge $. So $ A $ is $ g_{\mu_i}$-closed  with respect to $ \mu_j $ and hence $ A $ is $ \mu_i $-closed. This implies $ \{x\} $ is $ \mu_i $-open.  
\end{proof}
Note that in a GBT-space if each singleton   is either $ \mu_i $-open or $ \mu_i $-closed, $ i=1 $ or $ 2 $ then it does not imply that each $ g_{\mu_j} $-closed set with respect to $ \mu_i $ is $ \mu_j $-closed; $ i,j=1,2;i\not=j $,  as shown in the Example \ref{31}.\\

\begin{example}\label{31} Suppose  $ X = \{a, b, c, d\},  \mu_1=\{\emptyset,  \{a\}, \{b\}, \{a, b\}\}$ and $\mu_2=\{\emptyset, X $, \{a, b, d\}, $\{a, b, c\}\}$.  Then $ (X, \mu_1, \mu_2) $ is a GBT-space but not a bitopological space. Here each singleton is either $ \mu_1 $-open or $ \mu_2 $-closed. Now $ \{b, c, d\} $ is $ g_{\mu_2} $-closed with respect to $ \mu_1 $ but it is not $ \mu_2$-closed. So it is not pairwise $ T_\frac{1}{2}$.\end{example}

\begin{remark}\label{32} In view of Theorem \ref{28} and \ref{30}, it follows that a GBT-space  $ (X, \mu_1, \mu_2) $ is pairwise $ T_\frac{1}{2} $ if and only if both the conditions (i) and (ii) hold for any singleton $ \{x\}, x\in X $:

(i) $ \{x\} $ is either $ \mu_1 $-open or, $ \mu_2 $-closed

(ii) $ \{x\} $ is either $ \mu_2 $-open or, $ \mu_1 $-closed.
\end{remark}

\begin{theorem}\label{33} A GBT-space  $ (X, \mu_1, \mu_2) $ is pairwise $ T_\frac{1}{2} $ if and only if each singleton is either $ \mu_i $-open or $ \mu_i $-closed, $ i=1,2 $. Proof follows from Remark \ref{32}.\end{theorem}

\begin{theorem}\label{34} If a GBT-space  $ (X, \mu_1, \mu_2) $ is pairwise $ T_\frac{1}{2} $, then each singleton is either $ \mu_i $-open or $ \mu_j $-closed, $ i,j=1,2 $. The proof can be deduced from Remark \ref{32} .\end{theorem}

But the converse of Theorem \ref{34} may not be  true as seen from the Example \ref{35}. 
 
\begin{example}\label{35} Suppose  $ X = \{a, b, c\},  \mu_1=\{\emptyset, \{a\}, \{c\}, \{a, c\}\}$ and
$\mu_2=\{\emptyset, \{b\},  \{a, b\} 
\}$.  Then $ (X, \mu_1, \mu_2) $ is a generalized bitopological space but not a bitopological space. Here each singleton is either $ \mu_i $-open or  $ \mu_j $-closed. Now the subset $ \{b\} $ is $ g_{\mu_2} $-closed with respect to $ \mu_1 $, but it is not $ \mu_2 $-closed. Also $ \{a, c\} $ is $ g_{\mu_1} $-closed with respect to $ \mu_2 $, but it is not $ \mu_1 $-closed. So the GBT-space is not pairwise $ T_\frac{1}{2}$. Incidentally, it is pairwise $ T_1 $.
\end{example}

\begin{remark} In a GBT-space, it is  shown by the Examples \ref{35} and \ref{36} that there may not be a relation between pairwise $ T_1 $ and pairwise $ T_\frac{1}{2} $.\end{remark}

\begin{example}\label{36} Suppose  $ X = \{a, b, c, d\},  \mu_1=\{\emptyset, X, \{a\}, \{b\}, \{a, b\}, \{a, c, d\}, \{a, b, d\}\}$ and
$\mu_2=\{\emptyset, X, \{a\}, \{d\}, \{a, d\}, \{a, b, c\}$, 
$ \{a, b, d\}\}$.  Then $ (X, \mu_1, \mu_2) $ is a generalized bitopological space but not a bitopological space.  Clearly $ (X, \mu_1, \mu_2) $ is pairwise $ T_\frac{1}{2} $ by Theoem \ref{32}. But it is not pairwise $ T_1 $, since the pair of points $ a, c\in X, a\not=c $ contradict the definition of pairwise $ T_1 $.
\end{example}

\begin{remark}\label{37} From the Theorem \ref{15} and Remark \ref{32}, it follows that pairwise $ T_\frac{1}{2}$ implies pairwise $ T_0 $. But the converse may not be true as seen from the Example \ref{11}. \end{remark}

\section{\bf $ \lambda_{\mu_i}$-closed sets with respect to $ \mu_j $ and Pairwise $ \lambda $-closed sets}

\begin{definition} \label{38} A set $A$ in a  GBT-space $ (X, \mu_1, \mu_2) $ is said to be $ \lambda_{\mu_i}$-closed with respect to $ \mu_j $  if $ A=F_i\cap L_j $  where $F_i$ is a $ \mu_i $-closed set and $L_j$ is a $ \wedge_{\mu_j} $-set, $ i,j=1,2$ and $ i\not=j $. $A$  is said to be $ \lambda_{\mu_i}$-open with respect to $ \mu_j $ if $ X- A $ is  $ \lambda_{\mu_i}$-closed with respect to $ \mu_j $ or equivalently, $ A=V_i\cup M_j $, where $ V_i $ is a ${\mu_i}$-open set and $ M_j $ is a $ \vee_{\mu_j} $-set. \end{definition}

In a  GBT-space $ (X, \mu_1, \mu_2) $, it is clear that every $ \wedge_{\mu_j} $-set is $ \lambda_{\mu_i}$-closed with respect to $ \mu_j $ (taking $ X $ as a closed set) and every $ \mu_i $-closed set is $ \lambda_{\mu_i}$-closed with respect to $ \mu_j, i, j=1, 2; i\not=j $. But the converses are not always true as shown in the Example \ref{39}.
\begin{example}\label{39} 
Suppose  $ X = \{a, b, c\},  \mu_1=\{\emptyset, \{a\}\}$ and
$\mu_2=\{\emptyset, \{a, b\}\}$.  Then $ (X, \mu_1, \mu_2) $ is a generalized bitopological space but not a bitopological space. Now suppose $ B=\{b\} $. Then $ B $ is $ \lambda_{\mu_1}$-closed with respect to $ \mu_2 $. But $ B $ is not $ \wedge_{\mu_2} $-set since $ \{b\}\not=\{b\}_{\mu_2}^\wedge $. Also $ B $ is not $ \mu_1 $-closed.
\end{example}

\begin{theorem}\label{40} In a GBT-space, arbitrary union of $ \lambda_{\mu_i}$-open sets with respect to $ \mu_j $ is $ \lambda_{\mu_i}$-open, $ i, j=1, 2; i\not=j $.\end{theorem}
\begin{proof}
Suppose $ A_1,A_2,A_3,...$ are $ \lambda_{\mu_i}$-open sets with respect to $ \mu_j $ and $ \cup_{k} A_k=A $. So let $ A_k=U_{k_i}\cup M_{k_j}$ where $ U_{k_i} $'s are $ \mu_i $-open sets and $ M_{k_j}$'s are $ \vee _{\mu_j}  $-sets. Now $ A = \bigcup_{k} A_k=\bigcup_{k}(U_{k_i}\cup M_{k_j})=(\cup_{k}U_{k_i})\bigcup (\cup _{k}M_{k_j}) $. Clearly $\bigcup _{k} U_{k_i}$ is $ \mu_i $-open set and $\bigcup _{k} M_{k_j}$ is $ \vee_{\mu_j} $-set \cite{MS}. So $ A $ is $ \lambda_{\mu_i}$-open set with respect to $ \mu_j $.
\end{proof}

\begin{remark}\label{41}
It is observed in \cite{MS} that the collection of all $ \vee_{\mu_j} $-sets in a GT-space $ (X, \mu_j) $ forms a genaralized topology $ \kappa_j $, say. Also in a GBT-space $ (X, \mu_1, \mu_2) $,  the collection of all $ \lambda_{\mu_i}$-open sets with respect to $ \mu_j $ forms a genaralized topology containing $ (\mu_i\cup \kappa_j) $. $ i, j=1, 2; i\not=j $ which is  proved below in Corollary \ref{42}.
\end{remark}
\begin{corollary}\label{42} 
The collection of all $ \lambda_{\mu_i}$-open sets with respect to $ \mu_j $ in a GBT-space $ (X, \mu_1, \mu_2) $ forms a genaralized topology containing $ (\mu_i\cup \kappa_j) $. $ i, j=1, 2; i\not=j $. 
\end{corollary}

 Since  $ \emptyset $ is a $ \mu_i $-open set and  also a $ \vee_{\mu_j} $-set, proof follows from the Theorem \ref{40}.\\
\begin{lemma}\label{43} (c.f. \cite{JP})    
A set $A$ of a GBT-space  $ (X,  \mu_1, \mu_2)$ the following are equivalent: 

(i)  $A$ is $ \lambda_{\mu_i}$-closed with respect to $ \mu_j $, $ i, j=1, 2; i\not=j $.
(ii) $ A=P\cap A_{\mu_j}^\wedge, P $ is a $ \mu_i $-closed subset of $ X$, $ i, j=1, 2; i\not=j $.
(iii)  $ A= \overline{A_{\mu_i}}\cap L_j, L_j$ is a $ \wedge_{\mu_j} $-set, $ i, j=1, 2; i\not=j $.
(iv)  $ A=\overline{A_{\mu_i}}\cap A_{\mu_j}^\wedge $, $ i, j=1, 2; i\not=j $. 
\end{lemma}
The proof is simple, so is omitted.\\
In a  GBT-space, it is shown by citing Examples \ref{43A} and \ref{43B} that there may not exist a relation between $ \lambda_{\mu_i}$-closed sets  and $ g_{\mu_i}$-closed sets with respect to $ \mu_j, i, j=1, 2; i\not=j $.
\begin{example} \label{43A} 
$ g_{\mu_i}$-closed set with respect to $ \mu_j $ which is not $ \lambda_{\mu_i}$-closed with respect to $ \mu_j $.
Suppose  $ X = \{a, b, c, d\},  \mu_1=\{\emptyset, \{a\},  \{a, d\}\}$ and $\mu_2=\{\emptyset, \{b\}, \{b, d\} \}$.  Then $ (X, \mu_1, \mu_2) $ is a generalized bitopological space but not a bitopological space. Take $ C=\{c\} $, then $ C $ is $ g_{\mu_1}$-closed with respect to $ \mu_2 $ but not $ \lambda_{\mu_1}$-closed with respect to $ \mu_2 $ since $ C\not=\overline{C_{\mu_1}}\cap C_{\mu_2}^\wedge $. Again let $ B=\{b\} $, then $ B $ is  $ g_{\mu_2}$-closed with respect to $ \mu_1 $  but not $ \lambda_{\mu_2}$-closed with respect to $ \mu_1 $, since $ B\not=\overline{B_{\mu_2}}\cap B_{\mu_1}^\wedge $. 
\end{example}

\begin{example}\label{43B} $ \lambda_{\mu_i}$-closed with respect to $ \mu_j $ which is not $ g_{\mu_i}$-closed  with respect to $ \mu_j $.

Suppose  $ X = \{a, b, c, d\},  \mu_1=\{\emptyset, \{a\},  \{a, d\}\}$ and
$\mu_2=\{\emptyset, \{a, b\}, \{c\}, \{a, b, c\}\}$.  Then $ (X, \mu_1, \mu_2) $ is a generalized bitopological space but not a bitopological space. Suppose  $ A=\{a\} $. Then $ A $ is $ \lambda_{\mu_2}$-closed with respect to $ \mu_1 $ but not $ g_{\mu_2}$-closed  with respect to $ \mu_1 $ since $ \overline{A_{\mu_2}}\not\subset A_{\mu_1}^\wedge $. Again let $ B=\{b\} $. Then $ B $ is  $ \lambda_{\mu_1}$-closed with respect to $ \mu_2 $  but $ B $ is not $ g_{\mu_1}$-closed  with respect to $ \mu_2 $ since $ \overline{B_{\mu_1}}\not\subset B_{\mu_2}^\wedge $. 
\end{example}
\begin{definition}\label{44} 
In a  GBT-space $ (X, \mu_1, \mu_2) $, a set $A$ of  $ X $  is said to be pairwise $ \lambda $-closed  if $ A=(F_1\cap F_2)\cap (L_1\cap L_2) $  where $F_1, F_2$ are respectively  $ \mu_1,\mu_2 $-closed sets and $L_1,L_2$ are respectively $ \wedge_{\mu_1}, \wedge_{\mu_2} $-sets.  $A$ is said to be pairwise $ \lambda $-open if $ X-A $ is pairwise $ \lambda $-closed  or equivalently, $ A=(U_1\cup U_2)\cup ( M_1\cup M_2) $  where $U_1, U_2$ are respectively  $ \mu_1,\mu_2 $-open sets and $M_1,M_2$ are respectively $ \vee_{\mu_1}, \vee_{\mu_2} $-sets.
\end{definition}

\begin{lemma} \label{45}  For a set
 $A$ of a GBT-space  $ (X,  \mu_1, \mu_2) $  
the following are equivalent: 

(i)  $A$ is pairwise $ \lambda $-closed.

 (ii) $ A= (F_1\cap F_2)\cap  (A_{\mu_1}^\wedge \cap A_{\mu_2}^\wedge)$, where $ F_1, F_2 $ are respectively $ \mu_1, \mu_2 $-closed sets of $ X$.

(iii)  $ A= ( \overline{A_{\mu_1}}\cap \overline{A_{\mu_2}})\cap (L_1\cap L_2) $ where  $ L_1, L_2 $ are respectively $ \wedge_{\mu_1}, \wedge_{\mu_2} $-sets.

(iv)  $ A=(\overline{A_{\mu_1}}\cap \overline{A_{\mu_2}}) \cap (   A_{\mu_1}^\wedge\cap A_{\mu_2}^\wedge) $. 
\end{lemma} 
 
 The proof can be easily verified.
 \begin{remark}\label{46} 
From Lemma \ref{43} (iv) we can say that in a GBT-space $ (X,  \mu_1, \mu_2) $, a set $ A $ is $ \lambda_{\mu_i}$-closed with respect to $ \mu_j $ if $ A $ can be expressed as the intersection of all $ \mu_i $-closed sets and  $ \mu_j $-open sets containing it $(i,j=1,2; i\not=j)$. Also from Lemma \ref{45} (iv) we can say that $ A $ is pairwise $ \lambda $-closed  if $ A $ can be expressed as the intersection of all $ \mu_1 $-closed sets, $ \mu_2 $-closed sets and $ \mu_1 $-open sets, $ \mu_2 $-open sets containing it.
\end{remark}

Clearly in view of Lemma \ref{43} (iv), if a set $ A $ of $ (X, \mu_1, \mu_2) $ is $ \lambda_{\mu_1} $ or $ \lambda_{\mu_2} $ closed with respect to $ \mu_2 $ or $ \mu_1$ respectively  then it is pairwise $ \lambda $-closed, but the converse may not be true as seen from the Example \ref{46A}.

\begin{example} \label{46A}
Let  $ X = \{a, b, c, d\},  \mu_1=\{\emptyset, \{a, d\}, \{b, d\}, \{a, b, d\}\}$ and $\mu_2=\{\emptyset, \{a, b, c\}\}$.  Then $ (X, \mu_1, \mu_2) $ is a generalized bitopological space but not a bitopological space. Here $ \{a\} $ is neither $ \lambda_{\mu_1} $-closed with respect to $ \mu_2 $ nor $ \lambda_{\mu_2} $ closed with respect to  $ \mu_1$  but it is pairwise $ \lambda $-closed by Lemma \ref{45} (iv).
\end{example}

In a GBT-space $ (X,  \mu_1, \mu_2) $, we show by the  Example \ref{46B} that union of two $ \lambda_{\mu_i} $-closed sets with respect to $ \mu_j $ may not be  $ \lambda_{\mu_i} $-closed, $ i, j=1, 2; i\not=j $ and union of two pairwise $ \lambda $-closed sets may not be pairwise $ \lambda $-closed.

\begin{example} \label{46B}
Let  $ X = \{a, b, c, d\},  \mu_1=\{\emptyset, \{a\},  \{a, d, c\}\}$ and
$\mu_2=\{\emptyset, \{a, b, c\}\}$.  Then $ (X, \mu_1, \mu_2) $ is a generalized bitopological space but not a bitopological space. Take  $A= \{a\}$ and $B= \{d\} $, then $ A, B $ are $ \lambda_{\mu_2}$-closed sets with respect to $ \mu_1 $. But  $ A \cup  B=\{a,d\} $ is not $ \lambda_{\mu_2}$-closed with respect to $ \mu_1 $. Again we can verify that  $ A,B $ are also pairwise $ \lambda $-closed sets. But $ A\cup B = \{a,d\}$  is not pairwise $ \lambda $-closed. \end{example}

\begin{note}\label{47} In a GBT-space $ (X,  \mu_1, \mu_2) $, likewise Theorem \ref{40} it can be shown that   arbitrary union of pairwise $ \lambda $-open sets is pairwise $ \lambda $-open. So we can say that the collection of all pairwise $ \lambda $-open sets  forms a generalized topology on $ X$. \end{note}

\begin{theorem}\label{48} In a GBT-space $ (X, \mu_1, \mu_2) $, a set $ A $ of  $ X $ is $ \mu_i $-closed if and only if $ A $ is both $ g_{\mu_i} $-closed  and $ \lambda_{\mu_i} $-closed with respect to $ \mu_j $, $i, j=1, 2; i\not=j $. \end{theorem}
\begin{proof}
Obviously a $ \mu_i $-closed set is both $ g_{\mu_i} $-closed and $ \lambda_{\mu_i} $-closed with respect to $ \mu_j $. Conversely, suppose $ A $  is both $ g_{\mu_i} $-closed and $ \lambda_{\mu_i} $-closed with respect to $ \mu_j $. Then  by Remark \ref{9}, $ \overline{A_{\mu_i}}\subset A_{\mu_j}^\wedge $. But $ A $  is $ \lambda_{\mu_i} $-closed with respect to $ \mu_j $, then by Lemma \ref{43}, $ A= \overline{A_{\mu_i}}\cap  A_{\mu_j}^\wedge $. Therefore $ A= \overline{A_{\mu_i}}$, so $ A $  is $ \mu_i $-closed.  
\end{proof}

\begin{definition}\label{49} 
 $ A \subset (X,  \mu_1, \mu_2) $ is said to be $ \wedge_{\mu_1\mu_2} $-set if $ A =A_{\mu_1}^\wedge\cap A_{\mu_2}^\wedge$.
\end{definition}

\begin{note}\label{50} Clearly $ \wedge_{\mu_1\mu_2} $-set is  pairwise $ \lambda $-closed. But converse may not be true which is evident from Example \ref{17} taking the singleton $ \{c\} $ as the relevant set.\end{note}

\begin{theorem}\label{51} If $ (X,  \mu_1, \mu_2) $ is pairwise $ T_1 $, then every subset of $ X $ is $ \wedge_{\mu_1\mu_2} $-set.\end{theorem}
\begin{proof}
Suppose $ (X,  \mu_1, \mu_2) $ is pairwise $ T_1 $ and $ A\subset X $. Then by Theorem \ref{21}, each singleton is either $ \mu_1 $-closed or $ \mu_2 $-closed. So for $ x\in X, X-\{x\} $ is either $ \mu_1 $-open or $ \mu_2 $-open. Let $ A_1=  [x : x\in X, \{x\}$ is $ \mu_1 $-closed] and $ A_2= [x : x\in X, \{x\}$ is $ \mu_2 $-closed]. Then $ A_{\mu_1}^\wedge \subset A\cup A_2$ and $ A_{\mu_2}^\wedge \subset A\cup A_1$. So $ A \subset A_{\mu_1}^\wedge\cap A_{\mu_2}^\wedge \subset (A\cup A_1)\bigcap (A\cup A_2)=A\cup (A_1\cap A_2)=A$ which implies that $ A = A_{\mu_1}^\wedge\cap A_{\mu_2}^\wedge$. Hence $ A $ is  $ \wedge_{\mu_1\mu_2} $-set. 
\end{proof}

\begin{theorem} \label{52}If a GBT-space $ (X,  \mu_1, \mu_2) $ is pairwise $ T_\frac{1}{2} $, then every subset of $ X $ is pairwise $ \lambda $-closed.\end{theorem}
\begin{proof}
Suppose that the GBT-space $ (X,  \mu_1, \mu_2) $ is pairwise $ T_\frac{1}{2} $ and $ A\subset X $. Then by Remark \ref{32}, every singleton of $ X $ is either $ \mu_i $-open or $ \mu_j $-closed, $ i,j=1,2; i\not=j $. Let $ G_i= \{x: x\in X-A, \{x\} $ is $ \mu_i $-open, $ i=1,2\} $, $ H_j= \{x: x\in X-A, \{x\} $ is $ \mu_j $-closed, $ j=1,2\} $. So $ X-A= G_1\cup G_2\cup H_1\cup H_2$. Again let $ F_i=\cap [X-\{x\} : x\in G_i]=X-G_i, i=1,2 $ and $ L_j=\cap [X-\{x\} : x\in H_j]=X-H_j, j=1,2 $.

Note that $ F_i $ is $ \mu_i $-closed set and $ L_j $ is $ \wedge_{\mu_j} $-set.

Now, $ (F_1\cap F_2)\cap (L_1\cap L_2)=(X-G_1)\cap (X-G_2)\cap (X-H_1)\cap (X-H_2) = X-(G_1\cup G_2\cup H_1\cup H_2) = A$. Thus $ A $ is pairwise $ \lambda$-closed. 
\end{proof}

But the converse may not be true as seen from the Example \ref{35} where every subset is pairwise $ \lambda $-closed but the GBT-space is not pairwise $ T_\frac{1}{2} $.
\begin{definition}\label{53} A GBT-space $ (X,  \mu_1, \mu_2) $ is said to be pairwise $ \lambda $-symmetric if every pairwise $ \lambda $-closed set is both $ \lambda_{\mu_1} $-closed with respect to $ \mu_2 $ and $ \lambda_{\mu_2} $-closed with respect to $ \mu_1 $.
\end{definition}

\begin{theorem} \label{54} Pairwise $ T_1 $ and pairwise $ \lambda $-symmetric GBT-space  is   pairwise $ T_\frac{1}{2} $.

\end{theorem}
\begin{proof}
Suppose $ (X,  \mu_1, \mu_2) $ is pairwise $ T_1 $ and pairwise $ \lambda $-symmetric and $ \{x\} $ is not $ \mu_1 $-open. So $ X-\{x\}=A$, say, is not $ \mu_1 $-closed, but $ A $ is pairwise $ \lambda $-closed  by Theorem \ref{51} and Note \ref{50}. For pairwise $ \lambda $-symmetryness, $ A $ is $ \lambda_{\mu_1} $-closed with respect to $ \mu_2 $, so $ A=\overline{A_{\mu_1}}\cap A_{\mu_2}^\wedge $. Now $ \overline{A_{\mu_1}} = X $. Therefore $ A= A_{\mu_2}^\wedge\cap X=A_{\mu_2}^\wedge $. So $ A=X-\{x\} $ is a $ \mu_2 $-open set which implies $ \{x\} $ is a $ \mu_2 $-closed. Again if $ \{x\} $ is not $ \mu_2 $-open then we can prove in similar way that $ \{x\} $ is $\mu_1 $-closed. By Remark \ref{32}, the GBT-space is pairwise $ T_\frac{1}{2} $.
\end{proof}

\section{\bf    Pairwise $ \textit  T_\frac{1}{4}$, Pairwise $   \textit  T_\frac{3}{8}$  and pairwise $ \textit     T_\frac{5}{8}$ GBT-space}

\begin{definition}\label{55} A GBT-space $ (X,  \mu_1, \mu_2) $ is called pairwise $  T_\frac{1}{4}$    if and only if for every finite subset $P$ of $X$ and for every $ y\in X - P $,  there exists a set $A_y$ containing $P$ and disjoint from $\{y\}$  such that $A_y$ is either $ \mu_i $-open or $ \mu_j $-closed, $ i, j=1, 2 $.\end{definition}

\begin{theorem}\label{56} A GBT-space $ (X,  \mu_1, \mu_2) $ is pairwise $  T_\frac{1}{4}$ if and only if every finite subset of $ X $ is pairwise $ \lambda $-closed.
\end{theorem}
\begin{proof}
Let   $ (X,  \mu_1, \mu_2) $ be a pairwise $  T_\frac{1}{4}$ and $ P $ be a finite subset of $ X $. So for every $ y\in X-P $ there is a set $A_y$ containing $ P $ and disjoint from $ \{y\} $ such that $A_y$ is either $ \mu_i $-open or $ \mu_j $-closed, $ i, j=1, 2 $. Let $ L_i $ be the intersection of all such $ \mu_i $-open sets $A_y$ and $ F_j $ be the intersection of all such $ \mu_j $-closed sets $A_y$ as $ y $ runs over $ X-P $. Then $ L_i=(L_i)_{\mu_i}^\wedge $ and $ F_j $ is $ \mu_j $-closed  and $ P=F_1\cap F_2\cap L_1\cap L_2 $. Thus $ P $ is pairwise $ \lambda $-closed. 

Conversely, suppose $ P $ is a finite set of $ X,  y\in X-P $ and $ P $ is pairwise $ \lambda $-closed, then by Lemma \ref{45} (iv), $ P=\overline{P_{\mu_1}}\cap \overline{P_{\mu_2}}\cap P_{\mu_1}^\wedge \cap  P_{\mu_2}^\wedge $. If $ y $ belongs to none of the subsets $ \overline{P_{\mu_1}}, \overline{P_{\mu_2}}, P_{\mu_1}^\wedge,   P_{\mu_2}^\wedge $ then we are done. So $ y $ does not belong atleast one of the four subsets $ \overline{P_{\mu_1}}, \overline{P_{\mu_2}}, P_{\mu_1}^\wedge,   P_{\mu_2}^\wedge $. Suppose $ y\not\in \overline{P_{\mu_j}} $, then there exists a $ \mu_j $-closed set $A_y$(say) containing $ P $ such that $ \{y\}\not\subset A_y $. Again if $ y\not\in P_{\mu_i}^\wedge $, then there exists some $ \mu_i $-open set $A_y$ (say) containing $ P $ such that $ \{y\}\not\subset A_y $; $ i, j=1, 2 $.  Hence by Definition, $ (X,  \mu_1, \mu_2) $ is  pairwise $  T_\frac{1}{4}$.
\end{proof}

\begin{theorem}\label{57}  Every pairwise $T_\frac{1}{4}$ GBT-space $ (X,  \mu_1, \mu_2) $ is pairwise $  T_0 $. \end{theorem}
\begin{proof}
Suppose the GBT-space $ (X,  \mu_1, \mu_2) $ is pairwise  $T_\frac{1}{4}$ and $ x , y $ are two distinct points of $ X $.  Since the GBT-space is pairwise  $T_\frac{1}{4}$, then for every $ y\in X - \{x\}$ there exists a set $A_y$ such that $\{x\}\subset A_y$  and $ \{y\} \not \subset A_y $ where $ A_y$ is either $ \mu_i $-open or $ \mu_j $-closed, $ i, j=1, 2 $. Therefore by Theorem \ref{15}, the GBT-space is pairwise $T_0$.
\end{proof}

But converse may not be true as seen from Example \ref{17} taking $ \{a, b\} $ as finite set.
\begin{theorem}\label{58} A GBT-space $ (X,  \mu_1, \mu_2) $ is pairwise $  T_0 $ if and only if each singleton is pairwise $ \lambda $-closed. Proof is similar to the Theorem \ref{56}, so is omitted.
\end{theorem}

\begin{definition}\label{59} A GBT-space $ (X,  \mu_1, \mu_2) $ is called pairwise $  T_\frac{3}{8}$ if and only if  for every countable subset $P$ of $X$ and for every $ y\in X - P $,  there exists a set $A_y$ containing $P$ and disjoint from $\{y\}$  such that $A_y$ is either $ \mu_i $-open or $ \mu_j $-closed. $ i, j=1, 2 $.\end{definition}

\begin{theorem} \label{60}  A GBT-space $ (X,  \mu_1, \mu_2) $ is pairwise $T_\frac{3}{8}$ if and only if every countable subset of $X$ is pairwise $ \lambda $-closed. Proof is similar to  the Theorem \ref{56}, so is omitted.
\end{theorem}

 \begin{definition}\label{61} A GBT-space $ (X,  \mu_1, \mu_2) $ is called pairwise $  T_\frac{5}{8}$ if and only if for any subset $P$ of $X$ and for every $ y\in X - P $,  there exists a set $A_y$ containing $P$ and disjoint from $\{y\}$  such that $A_y$ is either $ \mu_i $-open or $ \mu_j $-closed. $ i, j=1, 2 $.\end{definition}

 \begin{theorem} \label{62}  A GBT-space $ (X,  \mu_1, \mu_2) $ is pairwise $T_\frac{5}{8}$ if and only if every subset $ E $ of $X$ is pairwise $ \lambda $-closed. Proof is similar to that of the Theorem \ref{56}, so is omitted.   \end{theorem}

\begin{remark} \label{63}It follows from Theorems \ref{52},  \ref{62},  \ref{60},    \ref{56} that every pairwise $ T_\frac{1}{2}$ GBT-space is pairwise $T_\frac{5}{8}$, every pairwise $T_\frac{5}{8}$ GBT-space is  pairwise $T_\frac{3}{8}$, every pairwise  $T_\frac{3}{8}$ GBT-space is pairwise $ T_\frac{1}{4}$. 
However, the converse of each implication may not be true which are substantiated respectively in the undermentioned examples.\end{remark}

  \begin{example}\label{64} Example of a pairwise $T_\frac{1}{4}$ GBT-space which is not pairwise $ T_\frac{3}{8}$.
 
  Suppose $X=R,  \mu_1 = \{\emptyset,  G_k $; where $ G_k$ are all countable subsets of rationals\}, $ \mu_2 = \{\emptyset, \{\sqrt{2}\cup (X-A) $, where $ A $ is finite and $ A\subset X $\}. Therefore $ (X, \mu_1, \mu_2) $ is a generalized bitopological space but not a bitopological space. Here any type of finite subset is pairwise $ \lambda $-closed. So it is pairwise $T_\frac{1}{4}$. Now suppose $ B $ is a countably infinite subset of irrationals such that $ \sqrt{2}\not\in B $. Then $ \overline{B_{\mu_1}}=P$, the set of all irrational numbers; $B_{\mu_2}^\wedge=B\cup \{\sqrt{2}\}, \overline{B_{\mu_2}}=X, B_{\mu_1}^\wedge = X $. Since  $ \overline{B_{\mu_1}}\cap \overline{B_{\mu_2}}\cap B_{\mu_1}^\wedge\cap   B_{\mu_2}^\wedge =B\cup \{\sqrt{2}\} \not=B $, so $ B $ is not pairwise $ \lambda $-closed  which implies the GBT-space is not pairwise $ T_\frac{3}{8}$.\end{example}

 \begin{example} Example of a pairwise $ T_\frac{3}{8}$ GBT-space which is not pairwise $T_\frac{5}{8}$.\end{example}
 
 It could be verified from the above Example \ref{64} by taking $ A $ is countable and $ B $ is a uncountably infinite subset of irrationals excluding $ \sqrt{2} $.
 \begin{example} 
Example of a pairwise $ T_\frac{5}{8}$ GBT-space which is not pairwise $T_\frac{1}{2}$. \\
Suppose $X=R,  \mu_1 = \{\emptyset,  G_k $; where $ G_k$ are all countable subsets of rationals\}, $ \mu_2 = \{\emptyset $, all cocountable subsets of $ X \} $. Therefore $ (X, \mu_1, \mu_2) $ is a generalized bitopological space but not a bitopological space. Here any type of subset viz. singleton, finite,  countable and uncountably infinite sets of $ R $ is pairwise $ \lambda $-closed. So it is pairwise $T_\frac{5}{8}$. Now take an uncountably infinite subset $ B \subset R $, then $ \overline{B_{\mu_2}}=X, B_{\mu_1}^\wedge=X $. So $ B $ is $ g_{\mu_2}$-closed with respect to $ \mu_1 $, but $ B $ is not $ \mu_2$-closed.  Hence the GBT-space is not pairwise $ T_\frac{1}{2}$. 
\end{example}

\begin{theorem}\label{65}
 Pairwise $ T_0 $ and pairwise $ R_0 $ GBT-space is pairwise $ T_1 $. 
\end{theorem}
\begin{proof}
Let the GBT-space $ (X,  \mu_1, \mu_2) $ be pairwise  $ T_0 $ and pairwise $ R_0 $ and $ x,  y \in X $ and $ x\not=y $. By Theorem \ref{18}, either $ x\not\in \overline{\{y\}_{\mu_1}} $ or $ y\not\in \overline{\{x\}_{\mu_2}} $. Suppose $ y\not\in \overline{\{x\}_{\mu_2}} $. Therefore $ y\in X - \overline{\{x\}_{\mu_2}}$,  a $ \mu_2 $-open set and $ x \not \in X  - \overline{\{x\}_{\mu_2}} $. Again for pairwise  $ R_0 $, $ \overline{\{y\}_{\mu_1}}\subset X  - \overline{\{x\}_{\mu_2}} $.  Therefore $ \overline{\{x\}_{\mu_2}}\cap \overline{\{y\}_{\mu_1}}=\emptyset$ which implies $ x \not\in \overline{\{y\}_{\mu_1}} $, so $ x\in X- \overline{\{y\}_{\mu_1}}$ and also $ y\not\in X- \overline{\{y\}_{\mu_1}}$, a $ \mu_1 $-open set. Thus the GBT-space is pairwise $ T_1 $.
\end{proof}

\begin{remark}\label{66} 
In a GBT-apace $ (X,  \mu_1, \mu_2) $, pairwise $ T_1 $ may not imply pairwise $ R_0 $ as it can be verified from the Example \ref{35} by taking the set $\{a\} $ as $ \mu_1 $-open.
\end{remark}

\begin{corollary} \label{67}
For a pairwise $ R_0 $ and pairwise $ \lambda $-Symmetric GBT-space $(X, \mu_1,\mu_2) $ the following are equivalent the proof of which can be done easily.

(1)  $(X, \mu_1,\mu_2) $ is pairwise $ T_0 $
   
(2)  $(X, \mu_1,\mu_2) $ is pairwise $ T_1 $
   
(3) $(X,\mu_1,\mu_2) $ is pairwise $ T_\frac{1}{2} $ 
   
(4)  $(X,\mu_1,\mu_2) $ is pairwise $ T_\frac{5}{8} $
   
(5) $(X,\mu_1,\mu_2) $ is pairwise $ T_\frac{3}{8} $ 
   
(6)  $(X,\mu_1,\mu_2) $ is pairwise $ T_\frac{1}{4} $.  
\end{corollary}   
\quad

\end{document}